\documentclass{amsart}[12pt]

\newtheorem{theorem}{Theorem}[section]

\newtheorem{proposition}[theorem]{Proposition}
\newtheorem{corollary}[theorem]{Corollary}

\theoremstyle{definition}
\newtheorem{definition}[theorem]{Definition}

\newtheorem{remark}[theorem]{Remark}

\usepackage{amscd,amssymb}

\begin{document}

\title
{A note on the $\mathbb{Z}_{2}\times \mathbb{Z}_{2}$-graded identities of $E \otimes E$ over a finite field}
\author{ Lu\'is Felipe Gon\c{c}alves Fonseca}

\address{Lu\'is Felipe Gon\c{c}alves Fonseca: Instituto de Ci\^encias Exatas e Tecnl\'ogicas \\
Universidade Federal de Vi\c{c}osa \\
Florestal, MG, Brazil. ZIP Code: 35690-000.}
\email{luisfelipe@ufv.br}

\date{}

\maketitle

\begin{abstract}
Let $F$ be a finite field of $char F = p$ and size $|F| = q$. Let $E$ be the unitary infinity dimensional Grassmann algebra. In this short note, we describe the $\mathbb{Z}_{2}\times \mathbb{Z}_{2}$-graded identities of $E_{k^{*}}\otimes E$, where $E_{k^{*}}$ is the Grassmann algebra with a specific $\mathbb{Z}_{2}$-grading. In the end, we discuss about the $\mathbb{Z}_{2}\times \mathbb{Z}_{2}$-graded GK-dimension of $E_{k^{*}}\otimes E$ in $m$ variables.

\end{abstract}

\section{Introduction}

Infinite-dimensional unitary Grassmann algebra $E$ is an important topic in PI-theory. Kemer \cite{Kemer} proved that every non-trivial variety of algebras over a field of characteristic zero is generated by the Grassmann hull of some finite dimensional algebra.

Many algebraist have studied the Grassmann algebra in a combinatorial way. In this approach, they investigate the identities and other properties.

In 1970's, Regev and Krakovski \cite{Regev 2} described the ordinary polynomial identities of $E$ over a field of characteristic zero. Years later, other mathematicians \cite{Ochir}, \cite{Giambruno 2} \cite{Regev} contributed to the study of polynomial identities of Grassmann algebra over fields of positive characteristic.

Beyond the ordinary case, the identities of Grassmann algebra have been studied in many other contexts in the recent years. One of this is the graded case. For this context, we can cite for example the contribution of \cite{Centrone-2}, \cite{Di Vincenzo 1}, \cite{Di Vincenzo 2}, \cite{Fonseca}, \cite{Giambruno 1}.

Other interesting task in PI-theory is the polynomial identities of square tensor Grassmann algebra. This algebra is one of the algebras covered by Kemer in Kemer’s Tensor Product Theorem.

In 1982, Popov \cite{Popov} described the ordinary polynomial identities of $E \otimes E$ over a field of characteristic zero. The description of the polynomial identities of $E \otimes E$ over a field of positive characteristic is still an open question. Despite this, there were progress in the graded case \cite{Centrone-1},\cite{Di Vincenzo 3}, \cite{Plamen} and \cite{Silva}. In this list, all authors worked over infinite fields of characteristic different from $2$. In \cite{Di Vincenzo 3}, the author worked over a field of characteristic zero.

In this paper, we combine the methods of Centrone and da Silva \cite{Centrone-1} and Fonseca \cite{Fonseca}. Our main goal is to describe the $\mathbb{Z}_{2}\times \mathbb{Z}_{2}$-graded polynomial identities of $E_{k^{*}} \otimes E$. Here $E_{k^{*}} \otimes E$ is viewed as an algebra over a finite field of characteristic different from $2$.

\section{Preliminaries: Grassmann Algebra}

Let $F$ be a finite field of characteristic $char (F) = p > 2$ and size $|F| = q$. In this paper, all vector spaces and all algebras will be over $F$.
We denote the set of the positive integers by $\mathbb{N}$. Let $\mathbb{Z}_{\geq 0} = \mathbb{N} \cup \{0\}$. We denote the set of the first $n$ positive integers by $\widehat{n}$.

Let $A$ be an algebra. We define the commutator of two elements $a_{1},a_{2}$ by $[a_{1},a_{2}] := a_{1}a_{2} - a_{2}a_{1}$. Inductively, the commutator of $n \geq 3$ elements $a_{1},\ldots,a_{n} \in A$ is defined by $[a_{1},\ldots,a_{n}] := [[a_{1},\ldots,a_{n-1}],a_{n}]$.

Let $G$ be a group. An algebra $A$ is graded when there exist subspaces $\{A_{g}\}_{g \in G}$ such that $A = \bigoplus_{g \in G}A_{g}$ and $A_{g}A_{h} \subset A_{gh}$ for all $g,h \in G$. An element $a \in A_{g}$ is called homogeneous of $G$-degree $g$. In this situation, let us denote: $\alpha(a) = g$.
\begin{definition}
Let $U$ be a vector space with an infinite countable basis $\newline \{e_{1},\ldots,e_{n},\ldots\}$. We denote the infinity-dimensional unitary Grassmann algebra of $U$ by $E$. We let $E^{*}$ denote the infinite-dimensional non-unitary Grassmann algebra of $U$.
\end{definition}
It is well known that $E = E_{0} \oplus E_{1}$, where:
\begin{center}
$E_{0} = span_{F}\{1_{E}, e_{i_{1}}.\ldots.e_{i_{2n}}| i_{1} < \ldots < i_{2n}, n \geq 1 \}$, \\
$E_{1} = span_{F}\{e_{i_{1}}.\ldots.e_{i_{2n+1}}| i_{1} < \ldots < i_{2n+1}, n \geq 0 \}$.
\end{center}
$(E_{0},E_{1})$ is called the canonical $\mathbb{Z}_{2}$-grading of $E$. The set
\begin{center}
$\mathcal{B} = \{1_{E}, e_{i_{1}}.\cdots.e_{i_{n}} | i_{1} < \ldots < i_{n}, n \geq 1 \}$
\end{center}
is called the canonical basis of $E$.

Let $b = e_{i_{1}}.\cdots.e_{i_{n}} \in \mathcal{B}$. The set $supp(b) = \{e_{i_{1}},\ldots,e_{i_{n}}\}$ is said to be the support of $b$. We denote the cardinality of $supp(b)$ by $l(b)$.

\begin{remark}

Let us consider a mapping $||.||: \{e_{1},e_{2},\ldots,\ldots\} \rightarrow \mathbb{Z}_{2}$ such that $||e_{j_{1}}.\ldots.e_{j_{m}}|| = ||e_{j_{1}}|| + \ldots + ||e_{j_{m}}||$. The mapping $||.||$ induces a $\mathbb{Z}_{2}$-grading on $E$. The map $||.||_{k^{*}}$ is defined by

$$||e_{i}||_{k^{*}}=\begin{cases}
1,&\mbox{for} \quad i = 1,\ldots,k \\
0, &\mbox{for}\quad i = k+1,\ldots.
\end{cases}
$$
\end{remark}

\section{Preliminaries: Free Algebra}

Let $T = \{t_{1},\ldots,t_{n},\ldots\}$ be an infinite countable set of variables. We will denote the ordinary free algebra generated by $T$ by $F\langle T \rangle$.

Let $V = \{v_{1},\ldots,v_{n},\ldots\}, W = \{w_{1},\ldots,w_{n},\ldots\}, X = \{x_{1},\ldots,x_{n},\ldots\}, Y = \{y_{1},\ldots,y_{n},\ldots\}$ be sets of infinite variables. $V,W,X,Y$ will denote variables of $\mathbb{Z}_{2}\times \mathbb{Z}_{2}$-degree $(0,0),(1,0),(0,1),(1,1),$ respectively.

\begin{definition}
Let $Z = V \cup W \cup X \cup Y$. The unitary free algebra freely generated by $Z$ will be denoted by $F\langle Z \rangle$. The algebra $F\langle Z \rangle$ has a natural $\mathbb{Z}_{2}\times \mathbb{Z}_{2}$-grading, where $\alpha(1) = (0,0)$ and:
\begin{center}
$F\langle Z \rangle_{(i,j)} = span_{F}\{z_{k_{1}}.\cdots.z_{k_{n}}| \alpha(z_{k_{1}}) + \ldots + \alpha(z_{k_{n}}) = (i,j)\}, \ \ (i,j) \in \mathbb{Z}_{2} \times \mathbb{Z}_{2}$.
\end{center}
\end{definition}

\begin{remark}
The algebra $F\langle Z \rangle$ has a natural $\mathbb{Z}_{2}$-grading too:
\begin{center}
$F\langle Z \rangle_{0} = F\langle Z \rangle_{(0,0)} \oplus F\langle Z \rangle_{(1,0)}$, $F\langle Z \rangle_{1} = F\langle Z \rangle_{(0,1)} \oplus F\langle Z \rangle_{(1,1)}$.
\end{center}
\end{remark}

\begin{remark}
We shall order the elements of $Z$. We consider
$v_{1} < v_{2} < \ldots < v_{n_{1}} < \ldots < w_{1} < \ldots < w_{n_{2}} < \ldots < x_{1} < \ldots < x_{n_{3}} < \ldots < y_{1} < \ldots < y_{n_{4}} < \ldots$
\end{remark}

We denote the variables of a polynomial $f \in F\langle Z \rangle$ by $var(f)$. Let $m \in F\langle Z \rangle$ be a monomial. Let $z \in Z$. We denote the number of occurrences of $z$ in $m$ by $deg_{z} m$.

A polynomial $f(z_{1},\ldots,z_{n}) \in F\langle Z \rangle$ is said to be a polynomial identity for $A$ when $f(a_{1},\ldots,a_{n}) = 0$ for all $a_{1}\in A_{\alpha(z_{1})},\ldots,a_{n}\in A_{\alpha(z_{n})}$. We denote the set of all polynomial identities of $A$ by $T_{\mathbb{Z}_{2}\times \mathbb{Z}_{2}}(A)$. Analogously, we can define the concept of ordinary polynomial identity of $A$.

Due to work of Regev (Lemma 1.2b and Corollaries 1.4 and 1.5, \cite{Regev}), it is known that

\begin{proposition}
\begin{enumerate}
\item The polynomial $t^{p}$ is a polynomial identity for $E^{*}$;
\item The image of the function $f: E \rightarrow E$, with rule $f(t) = t^{p}$, is $span_{F}\{1_{E}\}$. Particularly, if $a \in span_{F}\{1_{E}\}$ and $b \in E^{*}$, then $f(a + b) = a^{p} + b^{p} = a^{p}$;
\item  The polynomial $g(t) = t^{pq} - t^{p}$ is an ordinary polynomial identity for $E$.
\item If $f(t)$ is a polynomial identity of $E$, then $f(t)$ divides $g(t)$.
\end{enumerate}
\end{proposition}

An endomorphism $\phi: F\langle Z \rangle \rightarrow F\langle Z \rangle$ is said to be a $\mathbb{Z}_{2}\times \mathbb{Z}_{2}$-graded endomorphism when $\phi(F\langle Z \rangle_{g}) \subset F\langle Z \rangle_{g}$ for all $g \in \mathbb{Z}_{2}\times \mathbb{Z}_{2}$. An ideal $I \subset F\langle Z \rangle$ is called a $T_{\mathbb{Z}_{2}\times\mathbb{Z}_{2}}$-ideal when $\phi(I \cap F\langle Z\rangle_{g}) \subset F\langle Z\rangle_{g}$ for all graded endomorphism $\phi$ and for all $g \in \mathbb{Z}_{2} \times \mathbb{Z}_{2}$. The $T_{\mathbb{Z}\times \mathbb{Z}_{2}}$-ideal generated by $S$ is the intersection of all $T_{\mathbb{Z}\times \mathbb{Z}_{2}}$-ideals that contain $S$.
Let $S \subset F\langle Z \rangle$ be a non-empty set. We denote the $T_{\mathbb{Z}\times \mathbb{Z}_{2}}$-ideal generated by $\langle S by \rangle_{T_{\mathbb{Z}\times \mathbb{Z}_{2}}}$. We say that $S \subset F\langle Z \rangle$ is a basis for the graded identities of $A$ when $T_{\mathbb{Z}_{2}\times \mathbb{Z}_{2}}(A) = \langle S \rangle_{T_{\mathbb{Z}\times \mathbb{Z}_{2}}}$.

\section{Preliminaries: Tensor Square Grassmann Algebra}

Let $C = \{e_{i}\otimes 1_{E}, 1_{E}\otimes e_{j} | i = 1,2,\ldots; j = 1,2,\ldots\}$. The canonical basis $\mathcal{B'}$ of $E \otimes E$ is formed by $1_{E}\otimes 1_{E}$, the elements of $C$, and
\begin{center}
$\{e_{i_{1}}.\cdots .e_{i_{n}} \otimes e_{j_{1}}\cdots e_{j_{m}}| i_{1} < \ldots < i_{n}, j_{1} < \ldots < j_{m}, n,m = 1,2,\ldots \}$.
\end{center}

Notice that if $a_{1}\otimes b_{1},\ldots,a_{n}\otimes b_{n} \in \mathcal{B'}$, then $\prod_{i=1}^{n} a_{i}\otimes b_{i} \neq 0$ if and only if $supp(a_{i})\cap supp(a_{j}) = supp(b_{i})\cap supp(b_{j}) = \emptyset$ for all $i,j \in \widehat{n}$ distinct.

\begin{definition}
Let $a \otimes b, a'\otimes b' \in \mathcal{B'}$. Let us define:
\begin{enumerate}
\item $Supp(a\otimes b) = (supp(a) , supp(b))$ (tensor support of $a \otimes b$);
\item $l(a \otimes b) = l(a) + l(b)$ (support length).
\end{enumerate}
We say that $Supp(a \otimes b) \cap Supp (a' \otimes b') = \emptyset$ when $aba'b' \neq 0$.
\end{definition}


\begin{definition}
Let $a_{1}\otimes b_{1},\ldots, a_{n}\otimes b_{n}$ be distinct elements of $\mathcal{B'}$. Let $c = \sum_{i = 1}^{n}\lambda_{i} a_{i}b_{i} (\lambda_{i} \in F - \{0\})$. Let us define
\begin{enumerate}
\item $Supp-union(c) = (\cup_{j=1}^{n}(supp(a_{j})))\bigcup(\cup_{i=1}^{n}(supp(b_{i})))$,
\item $max-l(c) = max \{l(a_{1}) + l(b_{1}), \cdots, l(a_{n}) + l(b_{n}) \}$,
\item $max-ind(c) = \{i \in \widehat{n} | l(a_{i}) + l(b_{i}) = max-l(c)\}$,
\item $g-sum (c) = \sum_{i \in max-ind(c)} \lambda_{i}(a_{i}\otimes b_{i})$.
\end{enumerate}
\end{definition}

\section{Preliminaries: $p$-polynomials}

We now present the definition of $p$-polynomial.

\begin{definition}
A linear combination of monomials
\begin{center}
$f(v_{1},\ldots,v_{n}) = \sum_{j = 1}^{l} \alpha_{j}m_{j}(v_{1},\ldots,v_{n})$
\end{center}
is said to be a graded $p$-polynomial when $deg_{v_{i}} m_{j} \equiv 0 \ \ mod \ p, deg_{v_{i}} m_{l} \ < \ pq$ for all $i \in \widehat{n}$ and $j \in \widehat{l}$.
\end{definition}

In the same way, we can define ordinary $p$-polynomials. According to work of Bekh-Ochir and Rankin (\cite{Ochir}), if $p(t_{1},\ldots,t_{n})$ is a non-zero ordinary $p$-polynomial, there exist $\alpha_{1},\ldots,\alpha_{n} \in F$ such that $p(\alpha_{1}1_{E},\ldots,\alpha_{n}1_{E}) = \beta.1_{E} \neq 0$ (Corollary 3.1,\cite{Ochir}). It is not difficult to see that $p(\alpha_{1}1_{E}\otimes 1_{E},\ldots,\alpha_{n}1_{E}\otimes 1_{E}) = \beta.1_{E} \otimes 1_{E} \neq 0$. This fact implies the following corollary.

\begin{corollary}\label{p-polinomio}
Let $f(v_{1},\ldots,v_{n})$ be a non-zero ($\mathbb{Z}_{2}\times \mathbb{Z}_{2}$-graded) $p$-polynomial. There exists
$(\alpha_{1},\ldots,\alpha_{n}) \in F^{n}$ such that $f(\alpha_{1}1_{E}\otimes 1_{E},\ldots,\alpha_{n}1_{E}\otimes 1_{E}) = \beta 1_{E}\otimes 1_{E} \neq 0$.
\end{corollary}

\section{Preliminaries: Commutators}

Let $A = A_{0} \oplus A_{1}$ be a $\mathbb{Z}_{2}$-graded algebra. The tensor algebra $A \otimes E$ has a $\mathbb{Z}_{2}$-grading and a natural $\mathbb{Z}_{2}\times \mathbb{Z}_{2}$-grading defined as follows:
\begin{enumerate}
\item $(A \otimes E)_{0} = (A_{0}\otimes E_{0})\oplus(A_{1} \otimes E_{0})$, $(A \otimes E)_{1} = (A_{0}\otimes E_{1})\oplus(A_{1} \otimes E_{1})$.
\item $(A \otimes E)_{(0,0)} = (A_{0}\otimes E_{0})$, $(A \otimes E)_{(1,0)} =  (A_{1}\otimes E_{0})$, $(A \otimes E)_{(0,1)} = (A_{0}\otimes E_{1})$, $(A \otimes E)_{(1,1)} = (A_{1}\otimes E_{1})$.
\end{enumerate}

\begin{definition}\label{commutator}
Let $A = A_{0} \oplus A_{1}$ be a $\mathbb{Z}_{2}$-graded algebra. Let $a_{1},a_{2},\ldots,a_{n} \in A_{0}\cup A_{1}$. Let $b = b_{0} + b_{1}, b_{0} \in A_{0}, b_{1} \in A_{1}$. Let $c = c_{0} + c_{1}, c_{0} \in A_{0}, c_{1} \in A_{1}$. Let us define:
\begin{enumerate}
\item $[a_{1},a_{2}]_{\mathbb{Z}_{2}} := a_{1}a_{2} - (-1)^{\alpha(a_{1})\alpha(a_{2})}a_{2}a_{1}$;
\item $[c,d]_{\mathbb{Z}_{2}} := [b_{0},c_{0}]_{\mathbb{Z}_{2}} + [b_{0},c_{1}]_{\mathbb{Z}_{2}} + [b_{1},c_{0}]_{\mathbb{Z}_{2}} + [b_{1},c_{1}]_{\mathbb{Z}_{2}}$;
\item $[a_{1},\ldots,a_{n}]_{\mathbb{Z}_{2}} := [[a_{1},\ldots,a_{n-1}]_{\mathbb{Z}_{2}},a_{n}]_{\mathbb{Z}_{2}}$.
\end{enumerate}
\end{definition}

At light of Definition \ref{commutator}, we consider the commutator $[,]_{\mathbb{Z}_{2}}$ of elements of $A \otimes E$ or $F \langle Z \rangle$ viewing these $\mathbb{Z}_{2}\times \mathbb{Z}_{2}$-graded algebras as $\mathbb{Z}_{2}$-graded algebras.

\begin{remark}
Let $a_{1}\otimes b_{1},a_{2}\otimes b_{2} \in (\mathcal{B'})\cap (\cup_{(i,j)\in \mathbb{Z}_{2}\times \mathbb{Z}_{2}} (E_{k^{*}}\otimes E)_{(i,j)})$. We have

\begin{center}
$[a_{1}\otimes b_{1},a_{2}\otimes b_{2}]_{\mathbb{Z}_{2}} = [a_{1},a_{2}]\otimes b_{1}b_{2}$.
\end{center}
\end{remark}

By direct calculations, it is possible to verify that $[z_{1},z_{2},z_{3}]_{\mathbb{Z}_{2}}, z_{1},z_{2},z_{3} \in Z$ belongs to $T_{\mathbb{Z}_{2}\times \mathbb{Z}_{2}}(E_{k^{*}}\otimes E)$.

\begin{definition}
We denote the $T_{\mathbb{Z}_{2} \times \mathbb{Z}_{2}}$-ideal generated by the following type of identities
\begin{center}
$[z_{1},z_{2},z_{3}]_{\mathbb{Z}_{2}}, z_{1},z_{2},z_{3} \in Z$.
\end{center}
by $\mathcal{J}$.
\end{definition}

The next proposition is a result proved in \cite{Centrone-1}.

\begin{proposition}[Proposition 8, \cite{Centrone-1}]\label{proposicao-lucio}
Any $\mathbb{Z}_{2} \times \mathbb{Z}_{2}$-grading polynomial in $F \langle Z \rangle$ can be written, modulo $\mathcal{J}$, as a linear combination of polynomials of the type
\begin{center}
$z_{i_{1}}.\cdots.z_{i_{r}}[z_{j_{1}},z_{j_{2}}]_{\mathbb{Z}_{2}}.\cdots.[z_{j_{t-1}},z_{j_{t}}]_{\mathbb{Z}_{2}}$
\end{center}
with $z_{l} \in Z, i_{1} \leq i_{2} \leq \cdots \leq i_{r}$ and $j_{1} < \ldots < j_{t}$.
\end{proposition}

It is well known that $[t_{1}^{p},t_{2}] = - [t_{2},t_{1},\ldots,t_{1}]$ ($t_{1}$ appears $p$ times in the last commutator) is a consequence of $[t_{1},t_{2},t_{3}]$. So $[v_{1}^{p},v_{2}]_{\mathbb{Z}_{2}}$ belongs to $\mathcal{J}$. Consequently, we have the following result.

\begin{proposition}\label{0,0}
Let $m = v_{1}^{a_{1}}.\cdots.v_{n}^{a_{n}}$ be a monomial. Modulo $\mathcal{J}$,
\begin{center}
$m \equiv fv_{1}^{b_{1}}.\cdots.v_{n}^{b_{n}}$,
\end{center}
where $f$ is a monomial $p$-polynomial and $0 \leq a_{1},\ldots,a_{n} \leq p-1$.
\end{proposition}

\begin{definition}\label{definicao}
Let $u = h_{1}.h_{2}.h_{3}.h_{4}[z_{j_{1}},z_{j_{2}}]_{\mathbb{Z}_{2}}.\cdots.[z_{j_{t-1}},z_{j_{t}}]_{\mathbb{Z}_{2}}$ as in the statement
of Proposition \ref{proposicao-lucio}, where:
\begin{description}
\item $h_{1} = v_{i_{1}}^{a_{i_{1}}}.\cdots.v_{i_{n_{1}}}^{a_{i_{n_{1}}}}$, $h_{2} =  w_{j_{1}}^{b_{j_{1}}}.\cdots.w_{j_{n_{2}}}^{b_{j_{n_{2}}}}$,
\item $h_{3} = x_{l_{1}}^{c_{l_{1}}}.\cdots.x_{l_{n_{3}}}^{c_{l_{n_{3}}}}$, $h_{4} = y_{r_{1}}^{d_{r_{1}}}.\cdots.y_{r_{n_{4}}}^{d_{r_{n_{4}}}}$.
\end{description}
Let us define:
\begin{enumerate}
\item For $z \in Z$, $deg_{z} u:=$ the number of occurrences of $z$ in $u$,
\item $deg (u):= \sum_{z \in z} deg_{z} u$,
\item $beg(u):= h_{1}h_{2}h_{3}h_{4}$,
\item $beg_{VW}(u) := h_{1}h_{2}$,
\item $beg_{XY}(u) := h_{3}h_{4}$,
\item $exp_{XY}(u) := (c_{l_{1}},\ldots,c_{l_{n_{3}}}, d_{r_{1}},\ldots ,d_{r_{n_{4}}})$,
\item $end(u):= z_{j_{1}}z_{j_{2}}.\cdots.z_{j_{t-1}}z_{j_{t}}$.
\end{enumerate}
\end{definition}

\section{$\mathbb{Z}_{2}\times \mathbb{Z}_{2}$-graded identities of $E_{k^{*}} \otimes E$}

In (Definition 18, \cite{Centrone-1}), it is shown that for a field of positive characteristic, the following polynomials are $\mathbb{Z}_{2}\times \mathbb{Z}_{2}$-graded identities for $E_{k^{*}}\otimes E$.

\begin{definition}\label{identities-lucio-viviane}
The following polynomials are identities of $E_{k^{*}}\otimes E$
\begin{enumerate}
\item $(w_{1})^{p}$,
\item $(x_{1})^{p}z_{2}w_{3}$,
\item $(y_{1})^{p}z_{2}y_{3}$,
\item $(x_{1})^{p-1}[x_{1},z_{2}]_{\mathbb{Z}_{2}}$,
\item $(y_{1})^{p-1}[y_{1},z_{2}]_{\mathbb{Z}_{2}}$,
\item $[z_{1},z_{2},z_{3}]_{\mathbb{Z}_{2}}$,
\item $z_{1}z_{2}.\cdots.z_{k+1}, \ z_{1},\ldots,z_{k+1} \in W\cup Y$.
\end{enumerate}
\end{definition}
\begin{remark}
Notice that
\begin{enumerate}
 \item $(x_{1})^{p + 1} \in \langle (x_{1})_{1}^{p-1}[x_{1},z_{2}]_{\mathbb{Z}_{2}}\rangle_{\mathbb{Z}_{2}\times \mathbb{Z}_{2}} $,
 \item $ (y_{1})^{p + 1} \in \langle (y_{1})^{p-1}[y_{1},z_{2}]_{\mathbb{Z}_{2}} \rangle_{\mathbb{Z}_{2}\times \mathbb{Z}_{2}}$.
\end{enumerate}
\end{remark}

\begin{definition}
We denote the $T_{\mathbb{Z}_{2}\times \mathbb{Z}_{2}}$-ideal generated by the seven type of identities in Definition \ref{identities-lucio-viviane} by $\mathcal{H}$. We denote the $T_{\mathbb{Z}_{2}\times \mathbb{Z}_{2}}$ generated by the identity the seven identities in Definition \ref{identities-lucio-viviane} and the identity $v_{1}^{pq} - v_{1}^{p}$ by $\mathcal{I}$.
\end{definition}

\begin{definition}\label{hs}
Let $f$ be a polynomial as in the statement of Proposition \ref{proposicao-lucio}. The polynomial $f$ is said be of type $SE$ when there exist polynomials $h_{1},h_{2},h_{3},h_{4},h_{5}$ such that $f =  h_{1}h_{2}h_{3}h_{4}h_{5}$, where:

\begin{enumerate}
\item $h_{1} = v_{i_{1}}^{a_{i_{1}}}.\cdots.v_{i_{n_{1}}}^{a_{i_{n_{1}}}}, i_{1} < \ldots < i_{n_{1}}, 0 \leq a_{i_{1}},\ldots,a_{i_{n_{1}}} \leq p - 1$,
\item $h_{2} =  w_{j_{1}}^{b_{j_{1}}}.\cdots.w_{j_{n_{2}}}^{b_{j_{n_{2}}}}, j_{1} < \ldots < j_{n_{2}}, 0 \leq b_{j_{1}},\ldots,b_{j_{n_{2}}} \leq p - 1$,
\item $h_{3} = x_{l_{1}}^{c_{l_{1}}}.\cdots.x_{l_{n_{3}}}^{c_{l_{n_{3}}}}, l_{1} < \ldots < l_{n_{3}}, 0 \leq c_{l_{1}},\ldots,c_{l_{n_{3}}} \leq p$,
\item $h_{4} = y_{r_{1}}^{d_{r_{1}}}.\cdots.y_{r_{n_{4}}}^{d_{r_{n_{4}}}}, r_{1} < \ldots < r_{n_{4}}, 0 \leq d_{r_{1}},\ldots,d_{r_{n_{4}}} \leq p $.
\item $h_{5} = [z_{1},z_{2}][z_{3},z_{4}].\cdots.[z_{2n-1},z_{2n}],$ $n \geq 0$, $z_{1} < \ldots < z_{2n} (\mbox{if} \ \  n \geq 1)$.
\end{enumerate}
Furthermore, we say that $f$ is type $SSE$ when it is type $SE$ and the following three additional conditions hold
\begin{description}
\item $deg h_{2} + deg h_{4} \leq k$,
\item if $deg_{x_{i}} h_{3} = p$, then $deg_{x_{i}} end(h_{5}) = 0$,
\item if $deg_{y_{i}} h_{4} = p$, then $deg_{y_{i}} end(h_{5}) = 0$.
\end{description}
\end{definition}

Bearing in mind the Propositions \ref{proposicao-lucio}, \ref{0,0} and the Definition \ref{identities-lucio-viviane},
we have the following corollary.

\begin{corollary}\label{se}
Let $f \in F\langle Z \rangle$. Modulo $\mathcal{I}$, $f$ can be written as

\begin{center}
$f = f_{0} + \sum_{i = 1}^{n} f_{i}u_{i}$,
\end{center}
where $f_{0},f_{1},\ldots,f_{n}$ are $p$-polynomials and $u_{1},\ldots,u_{n} \in SS$. These elements of $SE$ are pairwise distinct
\end{corollary}

In (Proposition 20, \cite{Centrone-1}), the authors asserted that each element of type $SE$ (modulo $\mathcal{H}$) is a linear combination of elements of type $SSE$.

The next lemma is a refinement of Corollary \ref{se} and a consequence of (Proposition 20, \cite{Centrone-1}).

\begin{corollary}\label{sse}
Let $f \in F\langle Z \rangle$. Modulo $\mathcal{I}$, $f$ can be written as

\begin{center}
$f = f_{0} + \sum_{i = 1}^{n} f_{i}u_{i}$,
\end{center}
where $f_{0},f_{1},\ldots,f_{n}$ are $p$-polynomials and $u_{1},\ldots,u_{n} \in SSE$. These elements of $SSE$ are pairwise distinct.
\end{corollary}

\section{$SSE$ order}

We now introduced a variation of $SS$ Total Order defined in \cite{Fonseca} (see Definition 2.3 ). The symbol $<_{lex-lef}$ denotes the usual left lexicographical order.

\begin{definition}
Let $m \in \mathbb{N}$. Let $\mathcal{A} = (a_{1},\ldots,a_{m}) \in (\mathbb{Z}_{\geq 0})^{m}$. Let
\begin{center}
$ext(\mathcal{A}) = (b_{1},\ldots,b_{m},b_{m+1},b_{m+2},b_{m+3},\ldots,b_{m+n},\ldots)$,
\end{center}
where $b_{i} = a_{i}$ for $i \in \widehat{m},$ and $b_{i} = 0$ for $i \geq m + 1$.
\end{definition}

\begin{definition}\label{ext}
Let $m_{1},m_{2} \in \mathbb{N}$. Let
\begin{center}
$\mathcal{A} = (a_{1},\ldots,a_{m_{1}}) \in (\mathbb{Z}_{\geq 0})^{m_{1}},\mathcal{A'} = (a_{1}',\ldots,a_{m_{2}}') \in (\mathbb{Z}_{\geq 0})^{m_{2}}$ ; \\

$ext(\mathcal{A}) = (b_{1},\ldots,b_{m_{1}},\ldots), ext(\mathcal{A'}) = (b_{1}',\ldots,b_{m_{2}}',\ldots)$.
\end{center}
We say
that $ext(\mathcal{A}) \equiv_{(0,1)} ext(\mathcal{A'})$ when $(|b_{1}-b_{1}'|,\ldots,|b_{n} - b_{n}'|,\ldots) \in (\{0,1\})^{\infty}$.
\end{definition}

\begin{definition}
Let $ext(\mathcal{A}),ext(\mathcal{A'})$ be as in the statement of Definition \ref{ext}. Suppose that $ext(\mathcal{A}) \equiv_{(0,1)} ext(\mathcal{A'})$. We say that $ext(\mathcal{A}) <_{\equiv_{(0,1)}} ext(\mathcal{A'})$ when $a_{1}, |a_{1} - a_{1}'|$ are odd or there exists an integer $i > 1$ such that $a_{1} \equiv a_{1}',\ldots,a_{i-1} \equiv a_{i-1}' \ \ mod \ 2$, $a_{i}, |a_{i}-a_{i}'|$ are odd.
\end{definition}


\begin{definition}[SSE Order]
Let
\begin{center}
$m_{1} = z_{i_{1}}.\cdots.z_{i_{n_{1}}}[z_{j_{1}},z_{j_{2}}].\cdots.[z_{j_{2n_{2} - 1}},z_{j_{2n_{2}}}]$
\end{center}
and
\begin{center}
$m_{2} = z_{k_{1}}.\cdots.z_{k_{n_{3}}}[z_{l_{1}},z_{l_{2}}].\cdots.[z_{l_{2n_{4} - 1}},z_{l_{2n_{4}}}]$,
\end{center}
as in the statement of Definition \ref{hs}.
We say that $m_{1} < m_{2}$ when
\begin{enumerate}
\item If $deg m_{1} < deg m_{2}$;
\item If $deg m_{1} = deg m_{2},$ and $beg_{VW}(m_{1}) <_{lex-lef} beg_{VW}$;
\item If $deg m_{1} = deg m_{2},$ $beg_{VW}(m_{1}) = beg_{VW}(m_{2})$, $ext(exp_{XY}(m_{1})) \not\equiv_{(0,1)} ext(exp_{XY})(m_{2})$ and $beg_{XY}(m_{1}) <_{lex} beg_{XY}$,
\item If $deg m_{1} = deg m_{2},$ $beg_{VW}(m_{1}) = beg_{VW}(m_{2})$, $ext(exp_{XY}(m_{1})) \equiv_{(0,1)} ext(exp_{XY})(m_{2})$, $beg_{XY}(m_{1}) \neq beg_{XY}$ and \newline $ext(exp_{XY}(m_{1})) <_{\equiv_{(0,1)}} ext(exp_{XY}(m_{2}))$;
\item If $deg m_{1} = deg m_{2},$ $beg(m_{1}) = beg(m_{2})$ and $end(m_{1}) <_{lex-lef} end(m_{2})$.
\end{enumerate}
\end{definition}

\begin{definition}
Let $f = f_{0} + \sum_{i = 1}^{n} f_{i}u_{i}$ as in the Corollary \ref{sse}. The leading term of $f$ (which is denoted $LT(f)$) is the greatest element of the set $\{u_{1},\ldots,u_{n}\}$ (with respect the SSE order).
\end{definition}

\section{Some calculations}

We present some calculation on $E_{k^{*}}\otimes E$ in this section.
These calculations were done by \cite{Centrone-1} in other context (see Lemmas 13 and 17).

\begin{definition}
Let $z \in Z$. We denote an evaluation of $z$ in $(E_{k^{*}}\otimes E)_{\alpha(z)}$ by $\overline{z}$.
\end{definition}


\begin{proposition}\label{calculos}
Let $j,n,t \in \mathbb{Z}_{\geq 0}$. Let $\overline{z} = c\otimes d \in \mathcal{B'}$ such that $l(c)$ is odd.

\vspace{0.1cm}

Let

\vspace{0.1cm}

\begin{flushleft}
Type 1: $\overline{v_{1}} = \sum_{i = 1}^{t} e_{j + k + 2i - 1}e_{j + k + 2i}\otimes e_{n + 2i - 1}e_{n + 2i}$.
\end{flushleft}
\begin{flushleft}
 Type 2: $\overline{v_{2}} = e_{j + k+1}\otimes e_{n + 1}e_{n + 2} + \sum_{i = 1}^{t-1} e_{j + k + 2i}e_{j + k + 2i + 1}\otimes e_{n + 2i + 1}e_{n + 2i + 2}$.
\end{flushleft}

\begin{flushleft}
Type 3: $\overline{w_{1}} =  \sum_{i=1}^{t}e_{i}e_{j+k+i}\otimes e_{n + 2i - 1}e_{n + 2i}, t \leq k$.
\end{flushleft}

\begin{flushleft}
Type 4: $\overline{w_{2}} =  e_{1}\otimes e_{n+1}e_{n+2} + \sum_{i = 1}^{t-1}e_{i + 1}e_{j+k+i}\otimes e_{n+ 2i + 1}e_{n+ 2i + 2}, t \leq k$.
\end{flushleft}

\begin{flushleft}
Type 5: $\overline{x_{1}} =  \sum_{i=1}^{t} e_{j + k + i}\otimes e_{n + i}.$
\end{flushleft}

\begin{flushleft}
Type 6: $\overline{x_{2}} = e_{j + k+1}e_{j +k+2}\otimes e_{n + 1} + \sum_{i=1}^{t-1} e_{j+k+i+2} \otimes e_{n + i + 1}$.
\end{flushleft}

\begin{flushleft}
Type 7: $\overline{y_{1}} = \sum_{i=1}^{t}e_{i}\otimes e_{n+i}, t \leq k$.
\end{flushleft}

\begin{flushleft}
Type 8: $\overline{y_{2}} = e_{1}e_{k+1+j}\otimes e_{n+1} + \sum_{i=1}^{t-1}e_{i+1}\otimes e_{n+t+1}, t \leq k$.
\end{flushleft}

\begin{flushleft}
Type 9: $\overline{v_{3}} = \alpha 1_{E}\otimes 1_{E} + \overline{v_{1}}$.
\end{flushleft}

\begin{flushleft}
Type 10: $\overline{v_{4}} = \alpha 1_{E}\otimes 1_{E} + \overline{v_{2}}$.
\end{flushleft}

\vspace{0.3cm}
We have the following calculations:
\vspace{0.3cm}

\begin{flushleft}
1.1: $(\overline{v_{1}})^{t} = t!(\prod_{i = 1}^{t}e_{j+ k+2i-1}e_{j + k+2i})\otimes(\prod_{i = 1}^{t} e_{n + 2i - 1}e_{n + 2i})$.
\end{flushleft}

\vspace{0.1cm}

\begin{flushleft}
2.1: $\overline{v_{2}}^{t-1}[\overline{v_{2}},\overline{z}] = 2(t-1)!(e_{j+k+1}(\prod_{i = 1}^{t-1} e_{j + k+ 2i}e_{j+ k+2i + 1})c)\otimes((\prod_{i = 1}^{t} e_{n + 2i - 1}e_{n + 2i})d)$.
\end{flushleft}

\vspace{0.1cm}

\begin{flushleft}
3.1: $(\overline{w_{1}})^{t} = t!(\prod_{i=1}^{t}e_{i}e_{j+k+i})\otimes (\prod_{i=1}^{t} e_{n + 2i - 1}e_{n + 2i})$.
\end{flushleft}

\vspace{0.1cm}

\begin{flushleft}
4.1: $(\overline{w_{2}})^{t-1}[\overline{w_{2}},\overline{z}]_{\mathbb{Z}_{2}} = 2(t-1)!(e_{1} (\prod_{i = 1}^{t-1}e_{i + 1}e_{j+k+i})c)\otimes ((\prod_{i = 1}^{t}e_{n+2i - 1}e_{n+2i})d)$.
\end{flushleft}

\vspace{0.1cm}

\begin{flushleft}
5.1: $(\overline{x_{1}})^{t} = t!(\prod_{i=1}^{t} e_{j + k + i})\otimes (\prod_{i=1}^{t} e_{n + i}).$
\end{flushleft}

\vspace{0.1cm}

\begin{flushleft}
5.2: $(\overline{x_{1}})^{t-1}[\overline{x_{1}},\overline{z}]_{\mathbb{Z}_{2}} = 2t!((\prod_{i=1}^{t} e_{j + k + i})c)\otimes ((\prod_{i=1}^{t} e_{n + i})d)$.
\end{flushleft}

\vspace{0.1cm}

\begin{flushleft}
6.1: $\overline{x_{2}}^{t} = (t-1)!(\prod_{i=1}^{t+1}e_{j + k+i})\otimes (\prod_{i=1}^{t}e_{n+i})$ if $t$ is odd. If $t$ is even, we have $(\overline{x_{2}})^{t} = 0$.
\end{flushleft}

\vspace{0.1cm}

\begin{flushleft}
6.2: $(\overline{x_{2}})^{t-1}[\overline{x_{2}},\overline{z}]_{\mathbb{Z}_{2}} = 2(t-1)!( (\prod_{i=1}^{t+1}e_{j + k+i})c)\otimes ((\prod_{i=1}^{t}e_{n+i})d)$ if $t$ is even, $(\overline{x_{2}})^{t-1}[\overline{x_{2}},\overline{z}]_{\mathbb{Z}_{2}} = 0$ if $t$ is odd.
\end{flushleft}

\vspace{0.1cm}

\begin{flushleft}
7.1: $(\overline{y_{1}})^{t} = t!(\prod_{i=1}^{t}e_{i})\otimes (\prod_{i=1}^{t}e_{n+i})$.
\end{flushleft}

\vspace{0.1cm}

\begin{flushleft}
7.2: $(\overline{y_{1}})^{t-1}[(\overline{y_{1}}),\overline{z}] = 2t!((\prod_{i=1}^{t}e_{i})c)\otimes ((\prod_{i=1}^{t}e_{n+i})d)$.
\end{flushleft}

\vspace{0.1cm}

\begin{flushleft}
8.1: $(\overline{y_{2}})^{t} = (t-1)!(e_{1}e_{k+1+j}(\prod_{i=1}^{t-1}e_{i+1}\otimes e_{n+t+1}))\otimes (\prod_{i=1}^{t}e_{n+i})$ if $t$ is odd. $(\overline{y_{2}})^{t} = 0$ if $t$ is even.
\end{flushleft}

\vspace{0.1cm}

\begin{flushleft}
8.2: $(\overline{y_{2}})^{t-1}[\overline{y_{2}},\overline{z}] = 2(t-1)!((e_{1}e_{k+1+j}(\prod_{i=1}^{t-1}e_{i+1}\otimes e_{n+t+1}))c)\otimes ((\prod_{i=1}^{t}e_{n+i})d)$ if $t$ is even. $(\overline{y_{2}})^{t-1}[\overline{y_{2}},\overline{z}] = 0$ if $t$ is odd.
\end{flushleft}

\vspace{0.1cm}

\begin{flushleft}
9.1: $g-sum[(\overline{v_{3}})^{t}] = (\overline{v_{1}})^{t}$.
\end{flushleft}

\vspace{0.1cm}

\begin{flushleft}
10.1: $g-sum(\overline{v_{4}})^{t-1}[\overline{v_{4}},\overline{z}] = (\overline{v_{2}})^{t}$.
\end{flushleft}

\end{proposition}

\begin{definition}
Let $n_{1},n_{2},n_{3},n_{4} \in \mathbb{N}$. Let us define:
\begin{center}
$Z_{n_{1},n_{2},n_{3},n_{4}} := \{v_{1},\ldots,v_{n_{1}}\}\cup \{w_{1},\ldots,w_{n_{2}}\} \cup \{x_{1},\ldots,x_{n_{3}}\} \cup \{y_{1},\ldots,y_{n_{4}}\}$,
\end{center}
\begin{center}
$f(Z_{n_{1},n_{2},n_{3},n_{4}}) := f(v_{1},\ldots,v_{n_{1}},w_{1},\ldots,w_{n_{2}},x_{1},\ldots,x_{n_{3}},y_{1},\ldots,y_{n_{4}})$.
\end{center}
\end{definition}

\begin{definition}\label{adequate}
Let $u(Z_{n_{1},n_{2},n_{3},n_{4}})$ be an element of $SSE$. Let $z \in var(u)$. Suppose that $t = deg_{z} u$. A graded endomorphism $\phi: F\langle Z_{n_{1},n_{2},n_{3},n_{4}}\rangle \rightarrow E_{k^{*}}\otimes E$ is said to be \textit{suitable} with respect $u$ when
\begin{description}
\item For all distinct $z_{1},z_{2} \in var(u)$, we have $Supp(\phi(z_{1}))\cap Supp(\phi(z_{2})) = \emptyset$.
\end{description}
Furthermore
\begin{enumerate}
\item If $deg_{z} end(u) = 0$ and $z \in V$, then $\overline{z}$ is Type 1.

\item If $deg_{z} end(u) = 1$ and $z \in V$, then $\overline{z}$ is Type 2.

\item If $deg_{z} end(u) = 0$ and $z \in W$, then $\overline{z}$ is Type 3.

\item If $deg_{z} end(u) = 1$ and $z \in W$, then $\overline{z}$ is Type 4.

\item If $deg_{z} end(u) = 0$ and $z \in X$, then $\overline{z}$ is Type 5 if $t$ is even, $\overline{z}$ is Type 6 if $t$ is odd.

\item If $deg_{z} end(u) = 1$ and $z \in X$, then $\overline{z}$ is Type 5 if $t$ is odd, $\overline{z}$ is Type 6 if $t$ is even.

\item If $deg_{z} end(u) = 0$ and $z \in Y$, then $\overline{z}$ is Type 7 if $t$ is even, $\overline{z}$ is Type 8 if $t$ is odd.

\item If $deg_{z} end(u) = 1$ and $z \in Y$, then $\overline{z}$ is Type 7 if $t$ is odd, $\overline{z}$ is Type 8 if $t$ is even.
\end{enumerate}
\end{definition}

\begin{definition}
Let $\phi: F\langle Z_{n_{1},n_{2},n_{3},n_{4}} \rangle \rightarrow E_{k^{*}}\otimes E$ be a suitable graded homomorphism. An element $a \in \mathcal{B'}$ is
called complete with respect $\phi$ if
\begin{center}
$Supp-union(a) = \cup_{z \in Z_{n_{1},n_{2},n_{3},n_{4}}} Supp-union(\phi(z))$.
\end{center}
\end{definition}

\begin{proposition}\label{proposicao do adequado}
Let $u(Z_{n_{1},n_{2},n_{3},n_{4}})$ be an element of $SSE$. Let \newline
$\phi: F\langle Z_{n_{1},n_{2},n_{3},n_{4}}\rangle \rightarrow E_{k^{*}}\otimes E$ be a suitable graded homomorphism with respect $u$. Then:
\begin{enumerate}
\item $\phi(u) = \alpha a, \alpha \in F - \{0\}$ and $a \in \mathcal{B'}$,
\item $a$ is complete with respect to $\phi$.
\end{enumerate}
\end{proposition}
\begin{proof}
It follows straightforward from Proposition \ref{calculos} and Definition \ref{adequate}.
\end{proof}

\begin{remark}\label{ultimo recado}
At light of Proposition \ref{proposicao do adequado}, it is possible to define a natural homomorphism $\psi: F\langle Z_{n_{1},n_{2},n_{3},n_{4}} \rangle \rightarrow E_{k^{*}}\otimes E$ as follows:

$$\psi(z)=\left\{\begin{array}{rc}
\phi(z),&\mbox{if}\quad z \in W\cup X \cup Y,\\
\alpha_{z}1_{E}\otimes 1_{E} + \phi(z), &\mbox{if}\quad z \in V.
\end{array}\right.
$$

$\psi$ is called the homomorphism associated with the suitable homomorphism $\phi$ and the $n$-tuple $(\alpha_{v_{1}},\ldots,\alpha_{v_{n_{1}}}) \in F^{n}$.
\end{remark}

Bearing in mind the Corollary \ref{p-polinomio} and the Proposition \ref{proposicao do adequado}, we have the following.

\begin{proposition}\label{descricao final}
Let $u(Z_{n_{1},n_{2},n_{3},n_{4}})$ be an element of $SSE$. Let $f(v_{1},\ldots,v_{n_{1}})$ be a non-zero $p$-polynomial.

There exist a suitable homomorphism $\phi: F\langle Z_{n_{1},n_{2},n_{3},n_{4}}\rangle \rightarrow E_{k^{*}}\otimes E$ and an $n_{1}$-tuple $(\alpha_{v_{1}}1_{E}\otimes 1_{E},\ldots,\alpha_{v_{n_{1}}}1_{E}\otimes 1_{E})$ such that

\begin{enumerate}
\item $f(\alpha_{v_{1}}1_{E}\otimes 1_{E},\ldots,\alpha_{v_{n_{1}}}1_{E}\otimes 1_{E}) = \alpha 1_{E} \otimes 1_{E} \neq 0$
\item $g-sum(\psi(u)) = \beta b, \beta \in F - \{0\}, b \in \mathcal{B'}$,
\end{enumerate}
where $a$ is complete with respect $\phi$.
\end{proposition}

\section{Main theorem}

In the next theorem, we use an idea of (\cite{Fonseca} Theorem 3.7) and an idea of (\cite{Centrone-1} Theorem 21).

\begin{theorem}
The $T_{\mathbb{Z}_{2} \times \mathbb{Z}_{2}}$-ideals $\mathcal{I}$ and $T_{\mathbb{Z}_{2}\times \mathbb{Z}_{2}}(E_{k^{*}}\otimes E)$ are equal.
\end{theorem}
\begin{proof}
First of all, note that $\mathcal{I} \subset T_{\mathbb{Z}_{2}\times \mathbb{Z}_{2}}(E_{k^{*}}\otimes E)$. We will prove the reverse inclusion. Suppose for the sake of the contradiction that the reverse inclusion is false. In other words, there exists $f \in T_{\mathbb{Z}_{2}\times \mathbb{Z}_{2}}(E_{k^{*}}\otimes E) - \mathcal{I}$.

By Corollary \ref{sse}, $f \equiv f_{0} + \sum_{i=1}^{n}f_{i}u_{i}$, where $f_{0},\ldots,f_{n}$ are $p$-polynomials, $u_{1},\ldots,u_{n}$ are distinct elements of $SSE$. For simplicity, we assume that $u_{1}$ is $LT(f)$. Suppose without loss of generality that $var(f_{1}) = \{v_{1},\ldots,v_{n_{1}}\}$ and $var(u_{1}) = Z_{n_{1},n_{2},n_{3},n_{4}}$.

According to Corollary \ref{p-polinomio}, there exists an $n_{1}$-tuple $(\alpha_{v_{1}},\ldots,\alpha_{v_{n_{1}}}) \in F^{n_{1}}$ such that $f_{1}(\alpha_{v_{1}}1_{E}\otimes 1_{E},\ldots,\alpha_{v_{n_{1}}}1_{E}\otimes 1_{E}) \neq 0$. By Proposition
\ref{proposicao do adequado}, there exists a suitable graded homomorphism $\phi$ such that $\phi(u_{1}) = \beta a \neq 0, \beta \in F, a \in \mathcal{B'}$. Furthermore, $a$ is complete with respect $\phi$. Let $\psi$ be the homomorphism associated with $n_{1}$-tuple $(\alpha_{v_{1}},\ldots,\alpha_{v_{n_{1}}})$ and the suitable homomorphism $\phi$.

Let $u_{i} \neq u_{1}$. We claim que none summand of $\psi(u_{i})$ has support-length equal to $l(a)$. So, $g-sum(\psi(f)) = \beta a$. This fact will be a contradiction.

Indeed, $u_{i} < u_{1}$. Let $c$ be a summand of $\psi(u_{i})$. There are three cases to analyze.

\begin{description}
\item Case 1: $deg u_{i} < deg u_{1}$. In this situation, there exists a variable $z \in Z$ such that $deg_{z}u_{i} < deg_{z}u_{1}$. Bearing in mind the Proposition \ref{calculos} and the Definition \ref{adequate}, we conclude that $supp(\psi(z)) \not\subset supp(c).$ Consequently, $l(c) < l(a)$.

\item Case 2: $deg u_{1} = deg u_{1}$ and $u_{i} - u_{1}$ is not multi-homogeneous. In this situation, there exists a variable $z \in Z$ such that $deg_{z}u_{i} < deg_{z}u_{1}$. The conclusion is analogous to the case 1.
\item Case 3: $u_{1} - u_{i}$ is multi-homogeneous. Notice that if $beg(u_{1}) = beg(u_{i})$, then $end(u_{1}) = end(u_{i})$. Consequently $u_{i} = u_{1}$.
\subitem Subcase 3.1) $beg_{VW}(u_{i}) <_{lex-lef} beg_{VW}(u_{i})$. In this context, there exists a variable $z \in V \cup W$ such that $deg_{z}(beg(u_{i})) < deg_{z}(beg(u_{1}))$. By types 1 and 3 in Proposition \ref{calculos}, we have $\psi(u_{i}) = 0$.
\subitem Subcase 3.2) $beg_{VW}(u_{i}) = beg_{VW}(u_{i})$ and $beg_{XY}(u_{1}) \neq beg_{XY}(u_{i})$.
\subsubitem 3.2.1) Considering that $u_{1} - u_{i}$ is multi-homogeneous, it follows that $ext(exp_{XY}(u_{i})) \equiv_{(0,1)} ext(exp_{XY}(u_{1}))$. According to definition of SSE order, we have $ext(exp_{XY}(u_{i})) <_{(0,1)} \newline ext(exp_{XY}(u_{1}))$. So there exists a variable $z \cup X \cup Y$ such that $deg_{z}(beg(u_{1})$ and $|deg_{z}(beg(u_{1}) - deg_{z}(beg(u_{i}))|$ are odd. At light of types 6 and 8 in Proposition \ref{calculos}, we conclude that $\psi(u_{i}) = 0$
\end{description}

\end{proof}

\begin{remark}
In \cite{Centrone-1}, the authors calculate the $\mathbb{Z}_{2}\times \mathbb{Z}_{2}$- graded GK-dimension of $E_{k^{*}}\otimes E$ in $m$ variables over an infinite field of characteristic different of two.

Just as with the $\mathbb{Z}_{2}$-graded dimension of $E$ (see page 7, \cite{Centrone-3}), the free algebra $\frac{F\langle Z_{m,m,m,m}\rangle }{F\langle Z_{m,m,m,m}\rangle \cap T_{\mathbb{Z}_{2}\times\mathbb{Z}_{2}}(E_{k^{*}}\otimes E)}$ is finite dimensional. So, its $\mathbb{Z}_{2}\times \mathbb{Z}_{2}$ graded - Gelfand Kirilov dimension in $m$ variables is equal to $0$.
\end{remark}

\end{document}